\documentclass[11pt]{article}
\usepackage{amsmath, amsfonts, amsthm, amssymb, color}
\usepackage{graphicx}
\usepackage{float}
\usepackage{verbatim}

\allowdisplaybreaks

\numberwithin{equation}{section}

\newtheorem{lemma}{Lemma}[section]

\newtheorem{prop}[lemma]{Proposition}
\newtheorem{thm}[lemma]{Theorem}
\newtheorem{cor}[lemma]{Corollary}
\theoremstyle{definition}

\newtheorem{conjecture}[lemma]{Conjecture}
\newtheorem{metaconjecture}[lemma]{Metaconjecture}

\theoremstyle{remark}
\newtheorem{remark}[lemma]{Remark}

\def\O{\mathcal{O}}
\def\a{\mathbf{a}}
\def\b{\mathbf{b}}
\def\c{\mathbf{c}}

\def\p{\mathbf{p}}
\def\tmu{\tilde{\mu}}

\def\A{\mathcal{A}}

\numberwithin{equation}{section} \numberwithin{table}{section}

\title{Equidistribution results for self-similar measures}
\author{Simon Baker\\ \\
\emph{School of Mathematics,} \\ \emph{University of Birmingham,} \\ \emph{Birmingham,  B15 2TT, UK.} \\ Email: simonbaker412@gmail.com\\}

\date{\today}
\begin{document}
\maketitle

\begin{abstract}
	A well known theorem due to Koksma states that for Lebesgue almost every $x>1$ the sequence $(x^n)_{n=1}^{\infty}$ is uniformly distributed modulo one. In this paper we give sufficient conditions for an analogue of this theorem to hold for a self-similar measure. Our approach applies more generally to sequences of the form  $(f_{n}(x))_{n=1}^{\infty}$ where $(f_n)_{n=1}^{\infty}$ is a sequence of sufficiently smooth real valued functions satisfying some nonlinearity conditions. As a corollary of our main result, we show that if $C$ is equal to the middle third Cantor set and $t\geq 1$, then with respect to the natural measure on $C+t,$ for almost every $x$ the sequence $(x^n)_{n=1}^{\infty}$ is uniformly distributed modulo one. \\

\noindent \emph{Mathematics Subject Classification 2010}: 	11K06, 28A80.\\

\noindent \emph{Key words and phrases}: Self-similar measures, uniform distribution, powers of real numbers.

\end{abstract}

\section{Introduction}
A sequence $(x_n)_{n=1}^{\infty}$ of real numbers is said to be uniformly distributed modulo one if for every pair of real numbers $u,v $ with $0\leq u <v\leq 1$ we have 
\begin{equation}
\label{Uniform distribution}
\lim_{N\to\infty}\frac{\#\{1\leq n \leq N:x_n\bmod 1\in [u,v]\}}{N}=v-u.
\end{equation} The study of uniformly distributed sequences has its origins in the pioneering work of Weyl \cite{Weyl} from the early $20$th century. From these beginnings this topic has developed into an important area of mathematics, with many deep connections to Ergodic Theory, Number Theory, and Probability Theory. Generally speaking, it is a challenging problem to determine whether a given sequence of real numbers is uniformly distributed modulo one. Often the sequences one considers are of dynamical or number theoretic origins. For an overview of this topic we refer the reader to \cite{Bug}, \cite{KN}, and the references therein.

In this paper, we are interested in the distribution of the sequence $(x^n)_{n=1}^{\infty}$ modulo one for $x>1$. The study of these sequences dates back to the work of Hardy \cite{Hardy} and Pisot \cite{Pisot2,Pisot}. It is a difficult problem to describe the distribution of $(x^n)_{n=1}^{\infty}$ modulo one for specific values of $x$. It is still unknown whether there exists a transcendental $x>1$ such that $\lim_{n\to\infty}\inf_{m\in\mathbb{N}}|x^n-m|=0$. For some further background and recent results on the distribution of the sequence $(x^n)_{n=1}^{\infty}$ we refer the reader to \cite{A,AB,ABTY,Bak,Bug,BLR,BugMos,Dub} and the references therein. The generic behaviour of the sequence $(x^n)_{n=1}^{\infty}$ modulo one for $x>1$ is described by a well known theorem due to Koksma \cite{Koks}. This theorem states that for Lebesgue almost every $x>1$ the sequence $(x^n)_{n=1}^{\infty}$ is uniformly distributed modulo one. We are interested in determining whether analogues of Koksma's theorem hold for more general measures. More specifically, suppose $\mu$ is a Borel probability measure supported on $[1,\infty)$ that is defined ``independently" from the family of maps $\{f_{n}(x)=x^n\}_{n=1}^{\infty},$ we are interested in determining whether for $\mu$ almost every $x$ the sequence $(x^n)_{n=1}^{\infty}$ is uniformly distributed modulo one. Of course the important detail here is what exactly it means for a Borel probability measure to be independent from the family of maps $\{f_{n}(x)=x^n\}_{n=1}^{\infty}$. A natural family of measures to consider here are the self-similar measures generated by iterated function systems (defined in Section \ref{Preliminaries}). For our purposes an iterated function system will consist of a finite collection of contracting affine maps. Since for any $n\geq 2$ the map $f_{n}(x)=x^n$ is not affine, one could view the fact that self-similar measures are defined using affine maps as some sort of independence. As such the following conjecture seems plausible.

\begin{conjecture}
	\label{motivating conjecture}
Let $\mu$ be a non-atomic self-similar measure with support contained in $[1,\infty)$. Then for $\mu$ almost every $x$ the sequence $(x^n)_{n=1}^{\infty}$ is uniformly distributed modulo one.
\end{conjecture} In this paper we do not prove Conjecture \ref{motivating conjecture}. Our main contribution in this direction is Theorem \ref{power theorem} which lends significant weight to its validity. We conclude this introductory section by giving an overview of a number of related results that motivated the present work.

One of the most well known results from uniform distribution theory states that for any integer $b\geq 2,$ for Lebesgue almost every $x\in \mathbb{R}$ the sequence $(b^nx)_{n=1}^{\infty}$ is uniformly distributed modulo one (see \cite{Bug,KN}). In what follows we say that $x$ is $b$-normal if $(b^nx)_{n=1}^{\infty}$ is uniformly distributed modulo one. For an arbitrary Borel probability measure $\mu$ supported on $\mathbb{R}$ which is defined ``independently" from the dynamical system $x\to bx \bmod 1,$ it is natural to wonder whether $x$ is $b$-normal for $\mu$ almost every $x$. Just as above, the important detail here is what it means for a Borel probability measure to be independent from the dynamical system $x\to bx \bmod 1.$ The following metaconjecture encapsulates many important results in this direction.
\begin{metaconjecture}
Suppose $\mu$ is a Borel probability measure that is ``independent" from the dynamical system $x\to bx \bmod 1$. Then $\mu$ almost every $x$ is $b$-normal. 
\end{metaconjecture}  The first instances of this metaconjecture being verified are found in the papers of Cassels \cite{Cas} and Schmidt \cite{Sch}. These authors were motivated by a question of Steinhaus as to whether there exists an $x$ that is $b$-normal for infinitely many $b$ but not all $b$. They answered this question in the affirmative by proving that with respect to the natural measure on the middle third Cantor set, almost every $x$ is $b$-normal if $b$ is not a power of three. The underlying independence here comes from the middle third Cantor set being defined by similarities with contraction ratios equal to $1/3,$ and $b$ having a prime factor not equal to $3$. The current state of the art in this area are the following two theorems due to Hochman and Shmerkin \cite{HocShm}, and Dayan, Ganguly, and Weiss \cite{DGW}.

\begin{thm}{\cite[Theorem 1.4]{HocShm}}
	\label{HocShm}
Let $\{\varphi_i(x)=r_i x +t_i\}_{i\in \A}$ be an iterated function system satisfying the open set condition. Suppose $b\geq 2$ is such that $\frac{\log |r_i|}{\log b}\notin \mathbb{Q}$ for some $i\in \A,$ then for every fully supported\footnote{We say that a self-similar measure is fully supported if the corresponding probability vector $(p_i)_{i\in \A}$ satisfies $p_i>0$ for all $i\in \A$.} non-atomic self-similar measure $\mu$, $\mu$ almost every $x$ is $b$-normal.  
\end{thm}

\begin{thm}{\cite[Theorem 4]{DGW}}
	\label{DaGaWe}
Let $\{\varphi_i(x)=\frac{x}{b} +t_i\}_{i\in \A}$ be an iterated function system. Suppose $t_i-t_j\notin \mathbb{Q}$ for some $i,j\in \A$, then for every fully supported non-atomic self-similar measure $\mu$, $\mu$ almost every $x$ is $b$-normal.
\end{thm}

Some other important contributions in this area include the papers by Kaufman \cite{Kau}, and Queff\'{e}lec and Ramar\'{e} \cite{QR}, who constructed Borel probability measures supported on subsets of the badly approximable numbers whose Fourier transform converges to zero polynomially fast. Kaufman has also shown that such measures exist for the $\alpha$-well approximable numbers \cite{Kau2}. The results of Kaufman \cite{Kau}, and Queff\'{e}lec and Ramar\'{e} \cite{QR}, were recently extended by Jordan and Sahlsten to a more general class of measures \cite{JorSah}. Importantly, if the Fourier transform of a Borel probability measure converges to zero sufficiently fast (polynomial speed is sufficient), then it can be shown that almost every point with respect to this measure is $b$-normal for any $b\geq 2$. In fact, by a recent result of Pollington et al. \cite{Pol}, if the Fourier transform of a Borel probability measure converges to zero sufficiently fast, then for almost every $x,$ \eqref{Uniform distribution} holds for the sequence $(b^nx)_{n=1}^{\infty}$ with an explicit error term. 

Another related result was recently proved by Simmons and Weiss \cite{SimWei}. They proved that if $X\subset \mathbb{R}$ is a self-similar set satisfying the open set condition, then with respect to the natural measure on $X$, the orbit under the Gauss map ($x\to 1/x \bmod 1$) of almost every $x$ equidistributes with respect to the Gauss measure. Here the important point is that the natural measure on $X$ is defined independently from the dynamics of the Gauss map.

One of the challenges faced when addressing Conjecture \ref{motivating conjecture} is that, at least to the best of the author's knowledge, there is no dynamical system which effectively captures the distribution of $(x^n)_{n=1}^{\infty}$ modulo one. As such one cannot rely upon techniques from Ergodic Theory to prove this conjecture. Techniques from Ergodic Theory were previously applied with great success in the proofs of Theorem \ref{HocShm} and Theorem \ref{DaGaWe}. Instead of using these techniques, our approach will exploit the fact that the maps $f_{n}(x)=x^n$ are not affine for $n\geq 2$, and the fact that self-similar measures are defined using affine maps.

\section{Statement of results}

Our main contribution in the direction of Conjecture \ref{motivating conjecture} is the following theorem. 

\begin{thm}
\label{power theorem}
Let $\{\varphi_i(x)=rx+t_i\}_{i\in \A}$ be an equicontractive iterated function system satisfying the convex strong separation condition with self-similar set $X$ contained in $[1,\infty)$. Moreover let $(p_i)_{i\in \A}$ be a probability vector  satisfying $$\frac{1}{2}<\frac{-\sum_{i\in \A}p_i\log p_i}{-\log |r|}.$$ Then with respect to the self-similar measure $\mu$ corresponding to $(p_i)_{i\in \A}$, for $\mu$ almost every $x$ the sequence $(x^n)_{n=1}^{\infty}$ is uniformly distributed modulo one.
\end{thm}
We define what we mean by iterated function system, self-similar set, and what it means for an iterated function system to be equicontractive and to satisfy the convex strong separation condition in Section \ref{Preliminaries}. Importantly both of these conditions are satisfied by the iterated function system $\{\phi_{1}(x)=\frac{x+2t}{3},\phi_{2}(x)=\frac{x+2+2t}{3}\}$ for any $t\in \mathbb{R}.$ The self-similar set for this iterated function system is $C+t$ where $C$ is the middle third Cantor set. Using the fact that the restriction of the $\frac{\log 2}{\log 3}$-dimensional Hausdorff measure on $C+t$ coincides with the self-similar measure corresponding to the probability vector $(p_i)_{i=1}^2=(1/2,1/2),$ we see that Theorem \ref{power theorem} immediately implies the following corollary\footnote{The restriction of the $\frac{\log 2}{\log 3}$-dimensional Hausdorff measure on $C+t$ is given by $\mu(A)=\mathcal{H}^{\frac{\log 2}{\log 3}}(A\cap (C+t)).$ Here $\mathcal{H}^{\frac{\log 2}{\log 3}}$ is the $\frac{\log 2}{\log 3}$-dimensional Hausdorff measure. For more on Hausdorff measure see \cite{Fal1}. The restriction of the $\frac{\log 2}{\log 3}$-dimensional Hausdorff measure on $C+t$ can be thought of as the natural measure on $C+t$. }.

\begin{cor}
	Let $C$ be the middle third Cantor set. Then for any $t\geq 1,$ with respect to the restriction of the $\frac{\log 2}{\log 3}$-dimensional Hausdorff measure on $C+t,$ for almost every $x$ the sequence $(x^n)_{n=1}^{\infty}$ is uniformly distributed modulo one.
\end{cor}

Theorem \ref{power theorem} is implied by the following more general theorem which applies to a general class of functions.
\begin{thm}
\label{main theorem}
Let $\{\varphi_i(x)=rx+t_i\}_{i\in \A}$ be an equicontractive iterated function system satisfying the convex strong separation condition with self-similar set $X$ contained in $[1,\infty)$. Let $(f_n)_{n=1}^{\infty}$ be a sequence of functions satisfying the following properties:
\begin{enumerate}
	\item[A.] $f_{n}\in C^3(\text{conv}(X),\mathbb{R})$ for each $n.$\footnote{Here $\text{conv}(X)$ denotes the convex hull of $X$ and $C^3(\text{conv}(X),\mathbb{R})$ denotes the set of three times differentiable functions from $\text{conv}(X)$ to $\mathbb{R}$.}
	\item[B.] There exists $C_1,C_2>0$ such that for any $m,n$ with $m<n$ we have: $$|f_{n}'(x)-f_{m}'(x)|\leq C_1n^{C_2}x^{n-1}$$ for all $x\in \text{conv}(X)$.
	\item[C.] There exists $C_3>0$ such that for all $n$ sufficiently large, for any $m<n$ we have: $$|f_{n}''(x)-f_{m}''(x)|\geq C_{3}x^{n-2}$$ for all $x\in \text{conv}(X).$ 
	\item[D.] For any $m,n$ with $m<n$ we have either $$f_{n}'''(x)-f_{m}'''(x)\geq 0$$ for all $x\in \text{conv}(X),$ or $$f_{n}'''(x)-f_{m}'''(x)\leq 0$$ for all $x\in \text{conv}(X).$
\end{enumerate}  
Moreover let $(p_i)_{i\in \A}$ be a probability vector  satisfying $$\frac{1}{2}<\frac{-\sum_{i\in \A}p_i\log p_i}{-\log |r|}.$$ Then with respect to the self-similar measure $\mu$ corresponding to $(p_i)_{i\in \A}$, for $\mu$ almost every $x$ the sequence $(f_{n}(x))_{n=1}^{\infty}$ is uniformly distributed modulo one.
\end{thm}
\begin{remark}
	To see how Theorem \ref{power theorem} follows from Theorem \ref{main theorem} let $f_{n}(x)=x^n$ for all $n\geq 1$. Then for any $m,n$ with $m<n$ and $x\geq 1$ we have
	$$|f_{n}'(x)-f_{m}'(x)|=nx^{n-1}-mx^{m-1}\leq 2nx^{n-1}$$ and $$f_{n}'''(x)-f_{m}'''(x)=n(n-1)(n-2)x^{n-3}-m(m-1)(m-2)x^{m-3}\geq 0.$$ Moreover, if $n$ also satisfies $n\geq 2$ then $$|f_{n}''(x)-f_{m}''(x)|=n(n-1)x^{n-2}-m(m-1)x^{m-1}\geq (n(n-1)-m(m-1))x^{n-2}\geq x^{n-2}.$$ Therefore properties $B,$ $C$, and $D$ of Theorem \ref{main theorem} are satisfied by the sequence of functions $(f_{n}(x)=x^n)_{n=1}^{\infty}.$ Property $A$ of Theorem \ref{main theorem} is obviously satisfied by this sequence of functions. Therefore Theorem \ref{power theorem} follows from Theorem \ref{main theorem}.
\end{remark}

\begin{remark}
Note that we have deliberately phrased Theorem \ref{main theorem} with its application in the proof of Theorem \ref{power theorem} in mind. Theorem \ref{main theorem} still holds if the inequalities in property $B$ and property $C$ are replaced with the perhaps more natural inequalities: $$|f_{n}'(x)-f_{m}'(x)|\leq C_1n^{C_2}x^{n}$$ and $$|f_{n}''(x)-f_{m}''(x)|\geq C_{3}x^{n}.$$ These inequalities can be shown to be equivalent to those stated in property $B$ and property $C$ by altering the constants $C_{1}$ and $C_{3}$ appropriately. In particular, because $\text{conv}(X)$ is a compact subset of $[1,\infty)$ the extra powers of $x$ can be reconciled by altering the leading constant term.
\end{remark}
\begin{remark}
The hypotheses of Theorem \ref{main theorem} are satisfied by many sequences of functions. For instance we could take $f_{n}(x)=x^n+x^{n-1}+\cdots+x+1$ for all $n$. Alternatively we could fix a polynomial $g$ with strictly positive coefficients and let $f_{n}(x)=g(x)x^n$ for all $n$, or $f_{n}(x)=g(n)x^n$ for all $n$. Each of these sequences of functions satisfy the hypotheses of Theorem \ref{main theorem}. 

We can build further examples by taking a sequence of functions $(f_n)_{n=1}^{\infty}$ which satisfies the hypotheses of Theorem \ref{main theorem}, and a sequence of functions $(g_n)_{n=1}^{\infty}$ whose first and second derivatives grow subexponentially in $n$ and which also satisfies property $D$ with the same sign as $(f_n)_{n=1}^{\infty}.$ The sequence $(f_n+g_n)_{n=1}^{\infty}$ would then satisfy the hypotheses of Theorem \ref{main theorem}. To be more precise, we could take $(f_n)_{n=1}^{\infty}$ to be any sequence of functions satisfying the hypotheses of Theorem \ref{main theorem} where property $D$ is satisfied with positive sign, we then define a new sequence of functions $(h_{n}(x)=f_{n}(x)+n\log x)_{n=1}^{\infty}.$ The sequence $(h_n)_{n=1}^{\infty}$ then satisfies the hypotheses of Theorem \ref{main theorem} if $\text{conv}(X)\subset(1,\infty)$.


\end{remark}

The rest of this paper is organised as follows. In Section \ref{Preliminaries} we recall the necessary preliminaries from Fractal Geometry and the theory of uniform distribution. In Section \ref{Proof section} we prove Theorem \ref{main theorem}.

\section{Preliminaries}
\label{Preliminaries}
\subsection{Fractal Geometry}
We call a map $\varphi:\mathbb{R}\to\mathbb{R}$ a similarity if it is of the form $\varphi(x)=rx +t$ for some $r\in(-1,0)\cup(0,1)$ and $t\in \mathbb{R}$. We call a finite set of similarities $\{\varphi_i\}_{i\in \A}$ an iterated function systems or IFS for short. Here and throughout $\A$ denotes an arbitrary finite set. Given an IFS $\{\varphi_i(x)=r_ix+t_i\}_{i\in \A},$ we say that it is equicontractive if there exists $r\in(-1,0)\cup(0,1)$ such that $r_i=r$ for all $i\in \A$. Throughout this paper we will assume that if $\{\varphi_i\}_{i\in \A}$ is an equicontractive IFS then $r\in(0,1)$. For each of our theorems there is no loss of generality in making this assumption. This is because if $\{\varphi_i\}_{i\in \A}$ is an equicontractive IFS satisfying the convex strong separation condition, then $\{\varphi_i\circ \varphi_j\}_{(i,j)\in \A^2}$ is also an equicontractive IFS satisfying the convex strong separation condition and the contraction ratio is positive. Moreover, any self-similar measure for $\{\varphi_i\}_{i\in \A}$ can be realised as a self-similar measure for $\{\varphi_i\circ \varphi_j\}_{(i,j)\in \A^2}$.

An important result due to Hutchinson \cite{Hut} states that for any IFS $\{\varphi_i\}_{i\in \A},$ there exists a unique non-empty compact set $X$ satisfying $$X=\bigcup_{i\in \A}\varphi_i(X).$$ $X$ is called the self-similar set of $\{\varphi_i\}_{i\in \A}$. The middle third Cantor set and the von-Koch curve are well known examples of self-similar sets. Given a finite word $\a=(a_1,\ldots,a_M)\in \bigcup_{k=1}^{\infty}\A^k$ we let $$\varphi_{\a}:=\varphi_{a_1}\circ \cdots \circ \varphi_{a_M} \text{ and }X_{\a}:=\varphi_{\a}(X).$$ For distinct $\a,\b\in \A^M$ we let $$|\a\wedge\b|:=\inf\left\{1\leq k\leq M:a_k\neq b_k\right\}.$$
Given an IFS $\{\varphi_i\}_{i\in \A}$ and a probability vector $\p:=(p_i)_{i\in \A},$ there exists a unique Borel probability measure $\mu_{\p}$ satisfying \begin{equation}
\label{self-similar measure}
\mu_{\p}=\sum_{i\in \A}p_i\cdot \mu_{\p} \circ \varphi_{i}^{-1}.
\end{equation} We call $\mu_\p$ the self-similar measure corresponding to $\{\varphi_i\}_{i\in \A}$ and $\p.$ When the choice of $\p$ is implicit we simply denote $\mu_{\p}$ by $\mu$. For our purposes it is important that the relation \eqref{self-similar measure} can be iterated and for any $M\in\mathbb{N}$ the self-similar measure $\mu_{\p}$ satisfies
\begin{equation}
\label{iterated self-similar measure}
\mu_{\p}=\sum_{\a\in\A^{M}}p_{\a}\cdot \mu_{\p} \circ \varphi_{\a}^{-1},
\end{equation} where $p_{\a}=\prod_{k=1}^{M}p_{a_k}$ for $\a=(a_1,\ldots,a_M).$ Given a probability vector $\p$ we define the entropy of $\p$ to equal $$h(\p):=-\sum_{i\in \A}p_i\log p_i.$$ We emphasise that this quantity appears in the hypotheses of Theorem \ref{power theorem} and Theorem \ref{main theorem}. 

Many results in the study of self-similar sets require additional separation conditions on the IFS. Often one restricts to the case when the IFS satisfies the strong separation condition or the open set condition (see \cite{Fal1,Fal2}). In this paper we will require a slightly stronger separation condition that is still satisfied by many well known self-similar sets. Given an IFS $\{\varphi_i\}_{i\in \A}$, we say that $\{\varphi_i\}_{i\in \A}$ satisfies the convex strong separation condition if the convex hull of $X$ satisfies the following: $$\varphi_{i}(\textrm{conv}(X))\cap \varphi_{j}(\textrm{conv}(X))=\emptyset\qquad \forall i\neq j.$$ Iterated function systems satisfying the convex strong separation condition were also studied by Boore and Falconer in \cite{BooFal}. It is easy to construct iterated function systems satisfying the convex strong separation condition. For example, if we fix $r\in(0,1)$ and $\{t_i\}_{i=1}^n$ a finite set of real numbers satisfying $t_1<t_2<\cdots<t_{n},$ $1-r\leq t_{1},$ $r<t_{i+1}-t_{i}$, and $t_{n}\leq t_{1}+1-r$, then $\{\varphi_{i}(x)=rx+t_{i}\}_{i=1}^n$ is an IFS which satisfies the convex strong separation condition and whose self-similar set is contained in $[1,\infty).$

 To help with our exposition we state here an identity that will be used several times in our proof of Theorem \ref{main theorem}. Suppose $\{\varphi_i\}_{i\in \A}$ is an equicontractive IFS and $f\in C^{1}(\textrm{conv}(I),\mathbb{R})$. Then for any $\a\in \A^{M},$ it follows from the chain rule that the following equality holds
 \begin{equation}
 \label{chain rule}
 (f\circ \varphi_{\a})'(x)=r^{M}f'(\varphi_{\a}(x)).
 \end{equation}

\subsection{Uniform distribution}
To prove Theorem \ref{main theorem} we will make use of a well known criterion due to Weyl for uniform distribution in terms of exponential sums (see \cite[Theorem 1.2]{Bug} and \cite{Weyl}), and a result due to Davenport, Erd\H{o}s, and LeVeque (see \cite[Lemma 1.8]{Bug} and \cite{DEL}). Combining these results we may deduce the following statement.

\begin{prop}
	\label{Important prop}
Let $\mu$ be a Borel probability measure on $\mathbb{R}$ and $(f_n)_{n=1}^{\infty}$ be a sequence of continuous real valued functions. If for any $l\in \mathbb{Z}\setminus\{0\}$ the series
$$\sum_{N=1}^{\infty}\frac{1}{N}\int \left|\frac{1}{N}\sum_{n=1}^{N}e^{2\pi i lf_{n}(x)}\right|^2 d\mu$$
converges, then for $\mu$ almost every $x$ the sequence $(f_{n}(x))_{n=1}^{\infty}$ is uniformly distributed modulo one. 
\end{prop}
Proposition \ref{Important prop} is the tool that enables us to prove Theorem \ref{main theorem}. We will also rely on the following technical lemma due to van der Corput, for a proof of this lemma see \cite[Lemma 2.1.]{KN}

\begin{lemma}[van der Corput lemma]
	\label{van der Corput}
	Let $\phi:[a,b]\to\mathbb{R}$ be differentiable. Assume that $|\phi'(x)|\geq \gamma$ for all $x\in [a,b]$, and $\phi'$ is monotonic on $[a,b]$. Then $$\left|\int_{a}^be^{2\pi i\phi(x)}\, dx\right|\leq \gamma^{-1}.$$
\end{lemma}

\noindent \textbf{Notation.} Throughout this paper we will use $\exp(x)$ to denote $e^{2\pi i x}.$ Given two complex valued functions $f$ and $g$, we write $f=\O(g)$ if there exists $C>0$ such that $|f(x)|\leq C|g(x)|$ for all $x$. If the underlying constant depends upon some parameter $s,$ and we want to emphasise this dependence, we write $f=\O_{s}(g)$. Given an interval $I$ we let $|I|$ denote the Lebesgue measure of $I$.

\section{Proof of Theorem \ref{main theorem}}
\label{Proof section}
Let us now fix an IFS $\{\varphi_{i}\}_{i\in \A}$, a probability vector $\p,$ and a sequence of functions $(f_n)_{n=1}^{\infty}$ so that the hypotheses of Theorem \ref{main theorem} are satisfied. We let $\mu$ denote the self-similar measure corresponding to $\p$. Recall that $r$ denotes the contraction ratio of the elements of $\{\varphi_{i}\}_{i\in \A}$, and $X$ denotes the corresponding self-similar set. In what follows we let $$I:=\textrm{conv}(X).$$ Moreover, given a word $\a\in \cup_{k=1}^{\infty}\A^k$ we let $I_{\a}:=\varphi_{\a}(I).$

Recall that $X\subset [1,\infty)$. For technical reasons it is useful to restrict our arguments to subsets of $X$ that are a uniform distance away from $1$.  With this in mind we let the parameter $\kappa>0$ denote any small real number such that $1+\kappa\notin X.$ It follows from the convex strong separation condition that $\kappa$ exists and can be taken to be arbitrarily small.
Given such a $\kappa>0,$ we fix $\delta_{\kappa}>0$ to be any sufficiently small real number so that if we let
\begin{align*}
\Gamma_{\kappa}:=\max\Bigg\{ &r^{\delta_{\kappa}},\frac{1}{r^{2\delta_{\kappa}}}\left(\frac{e^{2(-h(\p)+\delta_{\kappa})}}{r}\right)^{\frac{\log(1+\kappa)}{-2\log r}},\frac{1+\delta_{\kappa}}{r^{3\delta_{\kappa}}}\left(\frac{e^{2(-h(\p)+\delta_{\kappa})}}{r}\right)^{\frac{\log(1+\kappa)}{-2\log r}},\\
&\frac{1+\delta_{\kappa}}{r^{3\delta_{\kappa}}}\left(\frac{e^{-h(\p)+\delta_{\kappa}}}{r^{\delta_{\kappa}}}\right)^{\frac{\log (1+\kappa)}{-2\log r}}\Bigg\},
\end{align*}then $$\Gamma_{\kappa}<1.$$ Such a $\delta_{\kappa}>0$ exists because of our underlying assumption $$\frac{1}{2}<\frac{h(\p)}{-\log r},$$ which is equivalent to $$\frac{e^{-2h(\p)}}{r}<1.$$ Moreover given such a $\kappa,$ and $\delta_{\kappa}$ chosen to be sufficiently small so that the above is satisfied, we fix $N_{\kappa}$ to be any sufficiently large natural number so that $$\max_{\a\in \A^{N_{\kappa}}}\sup_{x,y\in I_{\a}}\frac{x}{y}<1+\delta_{\kappa},$$ and for any $\a\in \A^{N_{\kappa}}$ we have either $$\sup I_{\a}<1+\kappa \textrm{ or }\inf I_{\a}>1+\kappa.$$ Such an $N_{\kappa}$ exists because $1+\kappa\notin X$ and $X$ is compact. 

Given a word $\c\in \cup_{k=1}^{\infty}\A^{k}$ we let $$\tmu_{\c}:=\frac{\mu|_{X_{\c}}}{\mu(X_{\c})}.$$ It is a consequence of the convex strong separation condition that $\tmu_{\c}=\mu\circ \varphi_{\c}^{-1}$. We will use this equality during our proof of Theorem \ref{main theorem}. 

It is a consequence of the following proposition that we can use Proposition \ref{Important prop} to prove Theorem \ref{main theorem}. 


\begin{prop}
	\label{decay prop}
Assume that $\{\varphi_i\}_{i\in \A}$, $\p,$ and $(f_n)_{n=1}^{\infty}$ satisfy the hypotheses of Theorem \ref{main theorem}. Then for any $\kappa>0$ such that $1+\kappa\notin X,$ there exists $\gamma:=\gamma(\kappa,\p)\in(0,1)$ such that for any $l\in \mathbb{Z}\setminus\{0\},$ $n>m,$ and $\c\in\A^{N_{\kappa}}$ satisfying $\inf I_{\c}>1+\kappa$, we have
	$$\int \exp(l(f_{n}(x)-f_{m}(x)))\, d\tmu_{\c}=\O_{\kappa,l}(\gamma^n).$$
\end{prop}
We now include the short argument explaining how Theorem \ref{main theorem} follows from Proposition \ref{decay prop}.

\begin{proof}[Proof of Theorem \ref{main theorem}]
It will be shown below that Proposition \ref{decay prop} implies that for any $\kappa>0$ such that $1+\kappa\notin X,$ if $\c\in\A^{N_{\kappa}}$ is such that $\inf I_{\c}>1+\kappa,$ then for $\tmu_{\c}$ almost every $x$ the sequence $(f_{n}(x))_{n=1}^{\infty}$ is uniformly distributed modulo one. It then follows from the definition of $N_{\kappa}$ and the self-similarity of $\mu$ (i.e. \eqref{iterated self-similar measure}), that this statement implies that for $\mu$ almost every $x>1+\kappa$ the sequence $(f_{n}(x))_{n=1}^{\infty}$ is uniformly distributed modulo one. Since there exists arbitrarily small $\kappa>0$ satisfying $1+\kappa\notin X,$  we may conclude that for $\mu$ almost every $x>1$ the sequence $(f_{n}(x))_{n=1}^{\infty}$ is uniformly distributed modulo one. Since $\mu(\{1\})=0$ Theorem \ref{main theorem} follows. To complete our proof of Theorem \ref{main theorem} it suffices to show that our initial statement is true.
	 
 Let us now fix $\kappa>0$ such that $1+\kappa\notin X$ and $\c\in \A^{N_{\kappa}}$ such that $\inf I_{\c}>1+\kappa.$ By Proposition \ref{Important prop}, to prove that for $\tmu_{\c}$ almost every $x$ the sequence $(f_{n}(x))_{n=1}^{\infty}$ is uniformly distributed modulo one, it suffices to show that for any $l\in \mathbb{Z}\setminus\{0\}$ we have
 \begin{equation}
 \label{suffices to show}
 \sum_{N=1}^{\infty}\frac{1}{N}\int \left|\frac{1}{N}\sum_{n=1}^{N}\exp(lf_{n}(x))\right|^2d\tmu_{\c}(x)<\infty.
 \end{equation}
 Expanding this expression we obtain
 \begin{align}
 \label{Expanded}
 &\sum_{N=1}^{\infty}\frac{1}{N}\int \left|\frac{1}{N}\sum_{n=1}^{N}\exp(lf_{n}(x))\right|^2d\tmu_{\c}(x)\nonumber\\
 =&\sum_{N=1}^{\infty}\left(\frac{1}{N^2}+\frac{1}{N^3}\sum_{\stackrel{1\leq n,m\leq N}{n\neq m}}\int \exp(l(f_{n}(x)-f_{m}(x)))\, d\tmu_{\c}\right).
 \end{align}
 The $1/N^2$ term appearing in \eqref{Expanded} does not affect the convergence properties of this series. As such it suffices to consider the remaining terms, which we can rewrite as 
 \begin{align}
\label{expand2}\sum_{N=1}^{\infty}\frac{1}{N^3}\sum_{\stackrel{1\leq n,m\leq N}{n\neq m}}\int \exp(l(f_{n}(x)-f_{m}(x)))d\tmu_{\c}
 =&\sum_{N=1}^{\infty}\frac{1}{N^3}\sum_{n=2}^{N}\sum_{m=1}^{n-1}\int \exp(l(f_{n}(x)-f_{m}(x)))\, d\tmu_{\c}\\
 +&\overline{\sum_{N=1}^{\infty}\frac{1}{N^3}\sum_{n=2}^{N}\sum_{m=1}^{n-1}\int \exp(l(f_{n}(x)-f_{m}(x)))\, d\tmu_{\c}}\nonumber.
 \end{align} Substituting the bound provided by Proposition \ref{decay prop} into \eqref{expand2} we obtain
 \begin{align*}
 \left|\sum_{N=1}^{\infty}\frac{1}{N^3}\sum_{\stackrel{1\leq n,m\leq N}{n\neq m}}\int \exp(l(f_{n}(x)-f_{m}(x)))d\tmu_{\c}\right|&= \O_{\kappa,l}\left(\sum_{N=1}^{\infty}\frac{1}{N^3}\sum_{n=2}^{N}\sum_{m=1}^{n-1}\gamma^n\right)\\
 &=\O_{\kappa,l}\left(\sum_{N=1}^{\infty}\frac{1}{N^3}\sum_{n=2}^{N}n\gamma^n\right)\\
 &=\O_{\kappa,l}\left(\sum_{N=1}^{\infty}\frac{1}{N^3}\right)\\
 &<\infty.
 \end{align*}In the penultimate line in the above, we have used the fact that $\sum_{n=2}^{N}n\gamma^n$ can be bounded above by a constant independent of $N$. We see that \eqref{suffices to show} now holds for any $l\in \mathbb{Z}\setminus\{0\}$ and our proof is complete.
\end{proof}

\subsection{Proof of Proposition \ref{decay prop}}
Throughout the rest of this section the parameter $\kappa$ is fixed. We assume that $\delta_{k}$ and $N_{\kappa}$ have been chosen so that the properties stated at the start of this section are satisfied. We also fix a word $\c\in \A^{N_{\kappa}}$ satisfying $\inf X_{\c}>1+\kappa$. We start our proof of Proposition \ref{decay prop} by defining several objects and collecting some useful estimates.

We let $x_0$ and $x_1$ be such that $$I_{\c}=[x_0,x_1].$$ Recall that by the definition of $N_{\kappa}$ we have 
\begin{equation}
\label{same growth}
\frac{x_1}{x_0}<1+\delta_{\kappa}.
\end{equation}Given $l\in \mathbb{Z}\setminus\{0\}$ and $n\in \mathbb{N}$ we define $$M=M(\c,l,\kappa,n):=\left \lfloor 1+\frac{\log 2\pi C_1|l||I|+C_2\log n+(n-1)\log x_1}{-2\log r}\right \rfloor+\delta_k n.$$ Importantly $M$ has the property that 
\begin{equation}
\label{M property}
r^{\delta_{\kappa} n +N_{\kappa}+2}\leq 2\pi C_{1}|l||I|n^{C_2}x_1^{n-1}r^{N_{\kappa}+2M}\leq r^{\delta_{\kappa} n +N_{\kappa}}.
\end{equation}Given $k\in \mathbb{N}$ we let $$B(k):=\left\{\a\in\A^{k}: p_{\a}\geq e^{k(-h(\p)+\delta_{\kappa})}\right\}.$$ It follows from a well known large deviation result due to Hoeffding \cite{Hoe} that for any $k\in\mathbb{N}$ there exists $\eta:=\eta(\kappa,\p)>0$ such that
\begin{equation}
\label{Hoeffding bound}
\sum_{\a\in B(k)}p_{\a}\leq e^{-\eta k}.
\end{equation} For $M$ as above we define $$G_{M}:=\left\{\a\in \A^{M}:(a_1,\ldots a_k)\notin B(k),\, \forall \lfloor \delta_{\kappa} M\rfloor \leq k \leq M\right\}.$$ It follows from \eqref{Hoeffding bound} and properties of geometric series that 
\begin{equation}
\label{good word bound}
\sum_{\stackrel{\a\in \A^{M}}{\a\notin G_{M}}}p_{\a}=\O_{\kappa}(e^{-\eta\delta_{\kappa} M}).
\end{equation} Given $m<n$ we define the function $$W_{M}(x):=\sum_{\a\in G_{M}}p_{\a}\exp(l(f_{n}(\varphi_{\c\a}(x))-f_{m}(\varphi_{\c\a}(x)))).$$ The proof of the following lemma is inspired by the proof of Lemma 6.1 from \cite{JorSah}. This lemma essentially allows us to bound from above the integral appearing in Proposition \ref{decay prop} by the $L^2$ norm of $W_{M}$ multiplied by a term that grows exponentially with $n$.


\begin{lemma}
\label{L2 bound}Let $m<n$ and $l\in \mathbb{Z}\setminus\{0\}.$ For $M$ as defined above we have
$$\left|\int \exp(l(f_{n}(x)-f_{m}(x)))\,d\tmu_{\c}\right|\leq   \frac{e^{M(-h(\p)+\delta_{\kappa})}}{|I|\cdot r^{M+2\delta_{\kappa}n}}\int_{I}|W_{M}(x)|^2\, dx + \O_{\kappa}(r^{\delta_{\kappa}n}+e^{-\eta\delta_{\kappa} M}).$$
\end{lemma}
\begin{proof}
Using first of all the relation $\tmu_{\c}=\mu\circ \varphi_{\c}^{-1},$ then \eqref{iterated self-similar measure}, we can rewrite our integral as follows:
\begin{align*}
\int \exp(l(f_{n}(x)-f_{m}(x)))\, d\tmu_{\c}&=\int \exp(l(f_{n}(\varphi_{\c}(x))-f_{m}(\varphi_{\c}(x))))\, d\mu\\
&=\int \sum_{\a\in \A^{M}}p_{\a}\exp(l(f_{n}(\varphi_{\c\a}(x))-f_{m}(\varphi_{\c\a}(x))))\, d\mu.
\end{align*}
Therefore it suffices to show that the latter integral satisfies the required bounds. By \eqref{good word bound} we see that 
\begin{equation}
\label{good word split}
\int \sum_{\a\in \A^{M}}p_{\a}\exp(l(f_{n}(\varphi_{\c\a}(x))-f_{m}(\varphi_{\c\a}(x))))\, d\mu = \int W_{M}(x)\, d\mu + \O_{\kappa}(e^{-\eta\delta_{\kappa} M}).
\end{equation}Let $$R_{M}:=\{\a\in G_{M}:\sup_{x\in X_{\a}}|W_{M}(x)|\geq 2r^{\delta_{\kappa}n}\}.$$ If $\a'\in R_{M},$ then by the mean value theorem, \eqref{chain rule}, property B for the sequence of functions $(f_n)_{n=1}^{\infty}$, and \eqref{M property}, for all $x\in I_{\a'}$ we have: 
\begin{align*}
&\quad |W_{M}(x)|\\
&\stackrel{M.V.T.}{\geq} 2r^{\delta_{\kappa}n}-\sup_{y\in I_{\a'}}|W_{M}'(y)|\cdot |I_{\a'}|\\
&\stackrel{\eqref{chain rule}}{=} 2r^{\delta_{\kappa}n} - \sup_{y\in I_{\a'}}\Big|\sum_{\a\in G_{M}}p_\a \cdot 2\pi i l r^{N_{\kappa}+M}(f_{n}'(\varphi_{\c\a}(y))-f_{m}'(\varphi_{\c\a}(y)))\exp(l(f_{n}(\varphi_{\c\a}(y))-f_{m}(\varphi_{\c\a}(y))))\Big|\cdot r^{M}|I|\\
&\stackrel{Property B}{\geq} 2r^{\delta_{\kappa}n} - \sup_{y\in I_{\a'}}\Big(\sum_{\a\in G_{M}}p_\a \cdot 2\pi |l| r^{N_{\kappa}+M}C_{1}n^{C_2}\varphi_{\c\a}(y)^{n-1}\Big)\cdot r^{M}|I| \\
&\geq  2r^{\delta_{\kappa}n} - \Big(\sum_{\a\in G_{M}}p_\a \cdot 2\pi |l| r^{N_{\kappa}+M}C_{1}n^{C_2}x_{1}^{n-1}\Big)\cdot r^{M}|I| \\
&\geq 2r^{\delta_{\kappa}n} - 2\pi C_1 |l||I|n^{C_2}x_{1}^{n-1}r^{N_{\kappa}+2M} \\
&\stackrel{\eqref{M property}}{\geq}  2r^{\delta_{\kappa}n}- r^{\delta_{\kappa}n+N_{\kappa}}\\
&\geq r^{\delta_{\kappa}n}.
\end{align*}
We have shown that 
\begin{equation}
\label{exponential lower bound}
|W_{M}(x)|\geq r^{\delta_{\kappa}n}
\end{equation} for all $x\in I_{\a'}$ for any $\a'\in R_{M}.$ Now notice that for any $\a\in R_{M}$ we have $$\int_{X_{\a}}|W_{M}(x)|\, d\mu \leq p_{\a}\quad\textrm{ and }\quad p_{\a}\leq e^{M(-h(\p)+\delta_{\kappa})}.$$ It follows that 
$$\sum_{\a\in R_{M}}\int_{X_{\a}}|W_{M}(x)|\, d\mu\leq \sum_{\a\in R_{M}} p_{\a}\leq \sum_{\a\in R_{M}}e^{M(-h(\p)+\delta_{\kappa})}.$$
Combining this upper bound with \eqref{exponential lower bound} we obtain 
\begin{align*}
\sum_{\a\in R_{M}}\int_{X_{\a}}|W_{M}(x)|\, d\mu&\leq \sum_{\a\in R_{M}}e^{M(-h(\p)+\delta_{\kappa})}\\
&=\frac{e^{M(-h(\p)+\delta_{\kappa})}}{|I|\cdot r^{M+2\delta_{\kappa}n}}\sum_{\a\in R_{M}}r^M|I|\cdot r^{2\delta_{\kappa}n}\\
&\leq\frac{e^{M(-h(\p)+\delta_{\kappa})}}{|I|\cdot r^{M+2\delta_{\kappa}n}}\sum_{\a\in R_{M}}\int_{I_{\a}}|W_{M}(x)|^2\, dx\\
&\leq \frac{e^{M(-h(\p)+\delta_{\kappa})}}{|I|\cdot r^{M+2\delta_{\kappa}n}}\int_{I}|W_{M}(x)|^2\, dx.
\end{align*}
In the last line we used that for distinct  $\a,\b\in\A^{M}$ the intervals $I_{\a}$ and $I_{\b}$ are disjoint. Using this upper bound, together with \eqref{good word bound} and the definition of $R_{M},$ we obtain 
\begin{align*}
\left|\int W_{M}(x)\,d\mu\right|\leq \int |W_{M}(x)|\,d\mu &=\sum_{\a\in G_{M}}\int_{X_\a} |W_{M}(x)|\,d\mu+\sum_{\stackrel{\a\in \A^{M}}{\a\notin G_{M}}}\int_{X_\a} |W_{M}(x)|\,d\mu\\
&\leq \sum_{\a\in G_{M}}\int_{X_\a} |W_{M}(x)|\,d\mu+\sum_{\stackrel{\a\in \A^{M}}{\a\notin G_{M}}}p_{\a}\\
&\leq \sum_{\a\in R_{M}}\int_{X_{\a}} |W_{M}(x)|\,d\mu+\sum_{\a\in G_{M}\setminus R_{M}}\int_{X_\a} |W_{M}(x)|\,d\mu+\O_{\kappa}(e^{-\eta\delta_{\kappa} M})\\
&\leq \sum_{\a\in R_{M}}\int_{X_{\a}} |W_{M}(x)|\,d\mu+\sum_{\a\in G_{M}\setminus R_{M}}p_{\a}\cdot 2r^{\delta_{\kappa}n}+\O_{\kappa}(e^{-\eta\delta_{\kappa} M})\\
&\leq  \sum_{\a\in R_{M}}\int_{X_{\a}} |W_{M}(x)|\,d\mu+2r^{\delta_{\kappa}n}+\O_{\kappa}(e^{-\eta\delta_{\kappa} M})\\
&\leq  \frac{e^{M(-h(\p)+\delta_{\kappa})}}{|I|\cdot r^{M+2\delta_{\kappa}n}}\int_{I}|W_{M}(x)|^2\, dx+\O_{\kappa}(r^{\delta_{\kappa}n}+e^{-\eta\delta_{\kappa} M}). 
\end{align*}Substituting this bound into \eqref{good word split} we obtain
$$\left|\int \sum_{\a\in \A^{M}}p_{\a}\exp(l(f_{n}(\varphi_{\c\a}(x))-f_{m}(\varphi_{\c\a}(x))))\, d\mu\right| \leq  \frac{e^{M(-h(\p)+\delta_{\kappa})}}{|I|\cdot r^{M+2\delta_{\kappa}n}}\int_{I}|W_{M}(x)|^2\, dx + \O_{\kappa}(r^{\delta_{\kappa}n} +e^{-\eta\delta_{\kappa} M})$$as required.

\end{proof}
To complete our proof of Proposition \ref{decay prop} it is necessary to obtain good upper bounds for $\int_{I} |W_{M}(x)|^2\, dx.$ These bounds are provided by the following lemma.

\begin{lemma}
\label{Lebesgue integral bound}
Let $m<n$ and $l\in \mathbb{Z}\setminus\{0\}.$ For $M$ as defined above we have 
$$\int_{I} |W_{M}(x)|^2\, dx=|I|\cdot e^{M(-h(\p)+\delta_{\kappa})}+\O_{\kappa,l}\left(\frac{1}{r^{M+\lfloor \delta_{\kappa} M\rfloor}x_0^{n}}+\frac{e^{M(-h(\p)+\delta_{\kappa})}}{r^{2M}x_0^{n}}\right).$$
\end{lemma}
\begin{proof}
We start by expanding $\int_I |W_{M}(x)|^2\, dx$:
\begin{align}
\label{L^2 expansion}
&\int_I |W_{M}(x)|^2\, dx\nonumber\\
=&|I|\sum_{\a\in G_{M}}p_{\a}^2 +\sum_{\stackrel{\a,\b\in G_{M}}{\a\neq \b}}p_{\a}\cdot p_{\b}\int_{I} \exp(l(f_{n}(\varphi_{\c\a}(x))-f_{m}(\varphi_{\c\a}(x))-f_{n}(\varphi_{\c\b}(x))+f_{m}(\varphi_{\c\b}(x))))\, dx\nonumber\\
\leq &|I|\cdot e^{M(-h(\p)+\delta_{\kappa})}+\sum_{\stackrel{\a,\b\in G_{M}}{\a\neq \b}}p_{\a}\cdot p_{\b}\int_{I} \exp(l(f_{n}(\varphi_{\c\a}(x))-f_{m}(\varphi_{\c\a}(x))-f_{n}(\varphi_{\c\b}(x))+f_{m}(\varphi_{\c\b}(x)))\, dx.
\end{align}
To bound the integral appearing in the summation in \eqref{L^2 expansion} we will use Lemma \ref{van der Corput}. Before doing this we demonstrate below that the hypotheses of this lemma are satisfied.\\

\noindent \textbf{Verifying the hypotheses of Lemma \ref{van der Corput}.}
Fix $\a,\b\in G_{M}$ such that $\a\neq\b$. Let   $$\phi(x):=l\left(f_{n}(\varphi_{\c\a}(x))-f_{m}(\varphi_{\c\a}(x))-f_{n}(\varphi_{\c\b}(x))+f_{m}(\varphi_{\c\b}(x))\right).$$ By \eqref{chain rule} we have $$\phi'(x)=r^{N_{\kappa}+M}l\left(f_{n}'(\varphi_{\c\a}(x))-f_{m}'(\varphi_{\c\a}(x))-f_{n}'(\varphi_{\c\b}(x))+f_{m}'(\varphi_{\c\b}(x))\right).$$ Define $$h_{n,m}(x):=f_{n}'(x)-f_{m}'(x).$$ Then $$\phi'(x)=r^{N_{\kappa}+M}l\left(h_{n,m}(\varphi_{\c\a}(x))-h_{n,m}(\varphi_{\c\b}(x))\right).$$ Applying the mean value theorem to the function $h_{n,m},$ we see that there exists $z\in I_{\c}$ such that 
\begin{equation}
\label{derivative}
\phi'(x)=r^{N_{\kappa}+M}l\left(\varphi_{\c\a}(x)-\varphi_{\c\b}(x)\right)\left(f_{n}''(z)-f_{m}''(z)\right).
\end{equation}
 It follows from the convex strong separation condition that there exists $c_0>0$ depending only on our underlying IFS such that
\begin{equation}
\label{lower bound1}
|\varphi_{\c\a}(x)-\varphi_{\c\b}(x)| \geq c_0r^{N_{\kappa}+|\a\wedge \b|}
\end{equation} for all $x\in I$. Using property C for our sequence of functions $(f_n)_{n=1}^{\infty},$ and the fact $z \in I_{\c}$ so $z\geq x_0$, it follows that
\begin{equation}
\label{lower bound2}
|f_{n}''(z)-f_{m}''(z)|\geq C_3z^{n-2}
\geq C_3x_{0}^{n-2}.
\end{equation}
Substituting \eqref{lower bound1} and \eqref{lower bound2} into \eqref{derivative}, we see that for all $x\in I$ we have
	\begin{equation}
	\label{Part one}
	|\phi'(x)|\geq c_0C_3lr^{2N_{\kappa}+M+|\a\wedge \b|}x_0^{n-2}.
	\end{equation}
	The right hand side of \eqref{Part one} is the value of $\gamma$ we will use in our application of Lemma \ref{van der Corput}. It remains to check that $\phi'$ satisfies the monotonicity hypothesis of Lemma \ref{van der Corput}. Differentiating $\phi'$ and applying \eqref{chain rule} we have 
	\begin{align*}
	\phi''(x)=r^{2(N_{\kappa}+M)}l\left(f_{n}''(\varphi_{\c\a}(x))-f_{m}''(\varphi_{\c\a}(x))-f_{n}''(\varphi_{\c\b}(x))+f_{m}''(\varphi_{\c\b}(x))\right).
	\end{align*}
	Applying the mean value theorem as above, this time to the function $f_{n}''(x)-f_{m}''(x),$ we may deduce that there exists $z\in I_{\c}$ such that $$\phi''(x)=r^{2(N_{\kappa}+M)}l(\varphi_{\c\a}(x)-\varphi_{\c\b}(x))(f_{n}'''(z)-f_{m}'''(z)).$$ By property D, for our sequence of functions $(f_n)_{n=1}^{\infty}$ we know that $f_{n}'''(z)-f_{m}'''(z)\geq 0$ for all $z\in I_{\c}$ or $f_{n}'''(z)-f_{m}'''(z)\leq 0$ for all $z\in I_{\c}$. What is more, it follows from the convex strong separation condition that the sign of $\varphi_{\c\a}(x)-\varphi_{\c\b}(x)$ is independent of $x$ and depends solely upon $\a$ and $\b$. Therefore we must have $\phi''(x)\leq 0$ for all $x\in I$ or $\phi''\geq 0$ for all $x\in I$. In either case $\phi'$ is monotonic, and we have shown that the monotonicity condition of Lemma \ref{van der Corput} is satisfied. \\
	
\noindent\textbf{Return to the proof of Lemma \ref{Lebesgue integral bound}.}	Taking the right hand side of \eqref{Part one} as our value of $\gamma$ in Lemma \ref{van der Corput}, we obtain
	\begin{equation}
\label{integral bound}	\int_{I} \exp(l(f_{n}(\varphi_{\c\a}(x))-f_{m}(\varphi_{\c\a}(x))-f_{n}(\varphi_{\c\b}(x))+f_{m}(\varphi_{\c\b}(x))))\, dx=\O_{\kappa,l}\left(\frac{1}{r^{M+|\a\wedge \b|}x_0^{n}}\right).
	\end{equation}
Substituting \eqref{integral bound} into the summation appearing in \eqref{L^2 expansion}, and using the definition of $G_{M},$ we see that the following holds:
\begin{align*}
&\sum_{\stackrel{\a,\b\in G_M}{\a\neq \b}}p_{\a}\cdot p_{\b}\int_{I} \exp(l(f_{n}(\varphi_{\c\a}(x))-f_{m}(\varphi_{\c\a}(x))-f_{n}(\varphi_{\c\b}(x))+f_{m}(\varphi_{\c\b}(x))))\, dx\\
=&\O_{\kappa,l}\left(\sum_{\a\in G_{M}}\sum_{\stackrel{\b\in G_M}{\a\neq \b}}\frac{p_{\a}\cdot p_{\b}}{r^{M+|\a\wedge\b|}x_{0}^n}\right)\\
=&\O_{\kappa,l}\left(\frac{1}{r^{M}x_0^{n}}\sum_{\a \in G_{M}}p_{\a}\sum_{k=1}^{M}\sum_{\stackrel{\b\in G_{M}}{|\a\wedge \b|=k}}\frac{p_{\b}}{r^{k}}\right)\\
=&\O_{\kappa,l}\left(\frac{1}{r^{M}x_0^{n}}\sum_{\a \in G_{M}}p_{\a}\sum_{k=1}^{M}\frac{\prod_{j=1}^kp_{\a_j}}{r^{k}}\right)\\
=&\O_{\kappa,l}\left(\frac{1}{r^{M}x_0^{n}}\sum_{\a \in G_{M}}p_{\a}\left(\sum_{k=1}^{\lfloor\delta_{\kappa} M\rfloor -1}\frac{\prod_{j=1}^kp_{\a_j}}{r^{k}}+\sum_{k=\lfloor\delta_{\kappa} M\rfloor}^{M}\frac{\prod_{j=1}^kp_{\a_j}}{r^{k}}\right)\right)\\
=&\O_{\kappa,l}\left(\frac{1}{r^{M}x_0^{n}}\sum_{\a \in G_{M}}p_{\a}\left(\sum_{k=1}^{\lfloor\delta_{\kappa} M\rfloor -1}\frac{1}{r^{k}}+\sum_{k=\lfloor\delta_{\kappa} M\rfloor}^{M}\frac{e^{k(-h(\p)+\delta_{\kappa})}}{r^{k}}\right)\right)\\
=&\O_{\kappa,l}\left(\frac{1}{r^{M}x_0^{n}}\sum_{\a \in G_{M}}p_{\a}\left(\frac{1}{r^{\lfloor \delta_{\kappa} M\rfloor}}+\frac{e^{M(-h(\p)+\delta_{\kappa})}}{r^{M}}\right)\right)\\
=&\O_{\kappa,l}\left(\frac{1}{r^{M+\lfloor \delta_{\kappa} M\rfloor}x_0^{n}}+\frac{e^{M(-h(\p)+\delta_{\kappa})}}{r^{2M}x_0^{n}}\right).
\end{align*}Substituting this bound into \eqref{L^2 expansion} we obtain 
$$\int_{I}|W_{M}(x)|^2\,dx=|I|\cdot e^{M(-h(\p)+\delta_{\kappa})}+\O_{\kappa,l}\left(\frac{1}{r^{M+\lfloor \delta_{\kappa} M\rfloor}x_0^{n}}+\frac{e^{M(-h(\p)+\delta_{\kappa})}}{r^{2M}x_0^{n}}\right)$$ as required.
\end{proof}


We are now in a position to prove Proposition \ref{decay prop} and in doing so complete our proof of Theorem \ref{main theorem}. 

\begin{proof}[Proof of Proposition \ref{decay prop}]
Assume that $m<n$. Combining Lemma \ref{L2 bound} and Lemma \ref{Lebesgue integral bound} we obtain
\begin{align}
\label{Mario}
&\left|\int \exp(l(f_{n}(x)-f_{m}(x)))\,d\tmu_{\c}\right|\nonumber\\
\leq &  \underbrace{\frac{e^{2M(-h(\p)+\delta_{\kappa})}}{r^{M+2\delta_{\kappa} n}}}_{(1)}+\O_{\kappa,l}\left(\underbrace{\frac{e^{M(-h(\p)+\delta_{\kappa})}}{r^{2M+2\delta_{\kappa}n+\lfloor \delta_{\kappa} M\rfloor}x_0^{n}}}_{(2)}+\underbrace{\frac{e^{2M(-h(\p)+\delta_{\kappa})}}{r^{3M+2\delta_{\kappa}n}x_0^{n}}}_{(3)}+\underbrace{r^{\delta_{\kappa} n}}_{(4)} +\underbrace{e^{-\eta \delta_{\kappa} M}}_{(5)}\right).
\end{align}	It remains to show that the terms $(1)-(5)$ decay to zero exponentially fast with respect to $n$. To do this it is useful to recall the definition of $\Gamma_{\kappa}$ and recall that we chose $\delta_{\kappa}$ in such a way that $\Gamma_{\kappa}<1$:
\begin{align*}
\Gamma_{\kappa}:=\max\Bigg\{ &r^{\delta_{\kappa}},\frac{1}{r^{2\delta_{\kappa}}}\left(\frac{e^{2(-h(\p)+\delta_{\kappa})}}{r}\right)^{\frac{\log(1+\kappa)}{-2\log r}},\frac{1+\delta_{\kappa}}{r^{3\delta_{\kappa}}}\left(\frac{e^{2(-h(\p)+\delta_{\kappa})}}{r}\right)^{\frac{\log(1+\kappa)}{-2\log r}},\\
&\frac{1+\delta_{\kappa}}{r^{3\delta_{\kappa}}}\left(\frac{e^{-h(\p)+\delta_{\kappa}}}{r^{\delta_{\kappa}}}\right)^{\frac{\log (1+\kappa)}{-2\log r}}\Bigg\}.
\end{align*} As we will see, most of the terms in \eqref{Mario} can be bounded in terms of $\Gamma_{\kappa}.$ To help with our exposition we treat each of the five terms described above individually.\\

\noindent \textbf{Bounding (1).} A useful inequality that follows from the definition of $M$ is that for $n$ sufficiently large we have
\begin{equation}
\label{M growth}
M\geq n\cdot\frac{\log x_1}{-2\log r}.
\end{equation}This inequality follows upon noticing that the floor term appearing in the definition of $M$ can be bounded below by $\frac{(n-1)\log x_1}{-2\log r}$ for $n$ sufficiently large, and then using the additional $\delta_{\kappa}n$ term. Applying \eqref{M growth}, the fact $x_1\geq 1+\kappa$, and the definition of $\Gamma_{\kappa},$ we see that the following holds for $n$ sufficiently large:
\begin{align}
\label{Bound 1}
\frac{e^{2M(-h(\p)+\delta_{\kappa})}}{r^{M+2\delta_{\kappa} n}}= \frac{1}{r^{2\delta_{\kappa} n}}\left(\frac{e^{2(-h(\p)+\delta_{\kappa})}}{r}\right)^{M}&\stackrel{\eqref{M growth}}{\leq} \left(\frac{1}{r^{2\delta_{\kappa} }}\left(\frac{e^{2(-h(\p)+\delta_{\kappa})}}{r}\right)^{\frac{\log x_1}{-2\log r}}\right)^n\nonumber\\
&\leq \left(\frac{1}{r^{2\delta_{\kappa} }}\left(\frac{e^{2(-h(\p)+\delta_{\kappa})}}{r}\right)^{\frac{\log (1+\kappa)}{-2\log r}}\right)^n\nonumber\\
&\leq \Gamma_{\kappa}^n
\end{align}

\noindent \textbf{Bounding (2).}
Applying \eqref{same growth}, \eqref{M property}, and \eqref{M growth} we have
\begin{align}
\label{Bound 2}
\frac{e^{M(-h(\p)+\delta_{\kappa})}}{r^{2M+2\delta_{\kappa}n+\lfloor \delta_{\kappa} M\rfloor}x_0^{n}}=&\frac{e^{M(-h(\p)+\delta_{\kappa})}}{r^{2M+2\delta_{\kappa}n+\lfloor \delta_{\kappa} M\rfloor}x_1^{n}}\left(\frac{x_1}{x_0}\right)^n \nonumber\\
&\stackrel{\eqref{same growth}}{\leq} \frac{e^{M(-h(\p)+\delta_{\kappa})}}{r^{2M+2\delta_{\kappa}n+\lfloor \delta_{\kappa} M\rfloor}x_1^{n}}(1+\delta_{\kappa})^n\nonumber \\
&\stackrel{\eqref{M property}}{=}\O_{\kappa,l}\left(\frac{n^{C_2}e^{M(-h(\p)+\delta_{\kappa})}}{r^{3\delta_{\kappa}n+\lfloor \delta_{\kappa} M\rfloor}}\left(1+\delta_{\kappa}\right)^n\right)\nonumber \\
&=\O_{\kappa,l}\left(n^{C_2}\left(\frac{1+\delta_{\kappa}}{r^{3\delta_{\kappa}}}\right)^n\left(\frac{e^{(-h(\p)+\delta_{\kappa})}}{r^{\delta_{\kappa} }}\right)^{M}\right)\nonumber\\
&\stackrel{\eqref{M growth}}{=}\O_{\kappa,l}\left(n^{C_2}\left(\frac{1+\delta_{\kappa}}{r^{3\delta_{\kappa}}}\left(\frac{e^{(-h(\p)+\delta_{\kappa})}}{r^{\delta_{\kappa} }}\right)^{\frac{\log x_1}{-2\log r}}\right)^n\right)\nonumber\\
&=\O_{\kappa,l}\left(n^{C_2}\left(\frac{1+\delta_{\kappa}}{r^{3\delta_{\kappa}}}\left(\frac{e^{(-h(\p)+\delta_{\kappa})}}{r^{\delta_{\kappa} }}\right)^{\frac{\log(1+\kappa)}{-2\log r}}\right)^n\right)\nonumber\\
&=\O_{\kappa,l}\left(n^{C_2}\Gamma_{\kappa}^n\right)\nonumber\\
&=\O_{\kappa,l}\left(\Gamma_{\kappa}^{n/2}\right).
\end{align}

\noindent \textbf{Bounding (3).} Repeating the argument used to bound (2) one can show that 
\begin{equation}
\label{Bound 3}
\frac{e^{2M(-h(\p)+\delta_{\kappa})}}{r^{3M+2\delta_{\kappa}n}x_0^{n}}=\O_{\kappa,l}\left(\Gamma_{\kappa}^{n/2}\right).
\end{equation} It is during this part of the proof that we use the fact that $$\frac{1+\delta_{\kappa}}{r^{3\delta_{\kappa}}}\left(\frac{e^{2(-h(\p)+\delta_{\kappa})}}{r}\right)^{\frac{\log(1+\kappa)}{-2\log r}}\leq \Gamma_{\kappa}.$$

\noindent \textbf{Bounding (4).} It is immediate from the definition of $\Gamma_{\kappa}$ that we have 
\begin{equation}
\label{Bound 4}
r^{\delta_{\kappa}n}\leq \Gamma_{\kappa}^n.
\end{equation}

\noindent \textbf{Bounding (5).} Applying \eqref{M growth} and the inequality $\log x_{1}\geq \log(1+\kappa),$ we see that the following holds for $n$ sufficiently large:
\begin{equation}
\label{Bound 5}
e^{-\eta \delta_{\kappa} M}\stackrel{\eqref{M growth}}{\leq} e^{\frac{\eta \delta_{\kappa}\log x_1}{2\log r}\cdot n} \leq e^{\frac{\eta \delta_{\kappa}\log (1+\kappa)}{2\log r}\cdot n}
\end{equation} 

We now let $$\gamma=\max\{\Gamma_{\kappa}^{1/2},e^{\frac{\eta \delta_{\kappa}\log(1+\kappa)}{2\log r}}\}.$$ Notice that $\gamma\in (0,1)$. Substituting \eqref{Bound 1}, \eqref{Bound 2}, \eqref{Bound 3}, \eqref{Bound 4}, and \eqref{Bound 5} into \eqref{Mario}, we obtain
$$\left|\int \exp(l(f_{n}(x)-f_{m}(x)))\,d\tmu_{\c}\right|=\O_{\kappa,l}\left(\gamma^n\right).$$ This completes our proof.
\end{proof}

\noindent \textbf{Acknowledgements.} The author would like to thank the anonymous referee for their useful
comments.


\begin{thebibliography}{100}
\bibitem{A} C. Aistleitner, \textit{Quantitative uniform distribution results for geometric progressions,} Israel J. Math. 204 (2014), no. 1, 155--197. 
\bibitem{AB} C. Aistleitner, S, Baker, \textit{On the pair correlations of powers of real numbers,} Israel J. Math. (to appear).
\bibitem{ABTY} C. Aistleitner, S, Baker, N. Technau, N. Yesha, \textit{Gap statistics and higher correlations for geometric progressions modulo one,} arXiv:2010.10355 
\bibitem{Bak} S. Baker, \textit{On the distribution of powers of real numbers modulo 1}, Unif. Distrib. Theory 10 (2015), no. 2, 67--75.
\bibitem{BooFal}  G. Boore, K. Falconer,\textit{ Attractors of directed graph IFSs that are not standard IFS attractors and their Hausdorff measure,} Math. Proc. Cambridge Philos. Soc. 154 (2013), no. 2, 325–-349.
\bibitem{Bug}Y. Bugeaud, \textit{Distribution modulo one and Diophantine approximation,}
Cambridge Tracts in Mathematics, 193.
Cambridge University Press, Cambridge, 2012.
\bibitem{BLR} Y. Bugeaud, L. Liao, M. Rams, \textit{ Metrical results on the distribution of fractional parts of powers of real numbers.} 
Proc. Edinburgh Math. Soc. 62 (2019), 505--521

\bibitem{BugMos}
Y. Bugeaud, N. Moshchevitin, \textit{On fractional parts of powers of real numbers close to $1$},
Math. Z. 271 (2012), no. 3--4,
627--637.	
\bibitem{Cas} J.W.S. Cassels, \textit{On a problem of Steinhaus about normal numbers,}
Colloq. Math. 7 (1959), 95--101.
\bibitem{DGW} Y. Dayan, A. Ganguly, B. Weiss, \textit{Random walks on tori and normal numbers in self similar sets,} arXiv:2002.00455 [math.DS].
\bibitem{DEL} H. Davenport, P. Erd\H{o}s, W.J. LeVeque, \textit{On Weyl's criterion for uniform distribution,} Michigan Math. J. 10 (1963), 311--314.
\bibitem{Dub}	A. Dubickas, \textit{On the powers of some transcendental numbers,}
Bull. Austral. Math. Soc. {\textbf{76}} (2007), no. 3,
433--440.
\bibitem{Fal1} K. Falconer, \textit{Fractal geometry. 
	Mathematical foundations and applications,} Third edition. John Wiley \& Sons, Ltd., Chichester, 2014. xxx+368 pp. ISBN: 978-1-119-94239-9. 
\bibitem{Fal2} K. Falconer, \textit{Techniques in fractal geometry,} John Wiley \& Sons, Ltd., Chichester, 1997. xviii+256 pp. ISBN: 0-471-95724-0
\bibitem{Hardy}
G. H. Hardy,	\textit{A problem of Diophantine approximation,}
J. Indian Math. Soc.  {\textbf{11}} (1919),
162--166.
\bibitem{Hoe} W. Hoeffding, \emph{Probability inequalities for sums of bounded random variables,} J. Amer. Statist. Assoc. 58 1963 13--30. 
\bibitem{HocShm} M. Hochman, P. Shmerkin, \textit{Equidistribution from fractal measures,}
Invent. Math. 202 (2015), no. 1, 427--479.
\bibitem{Hut}  J. Hutchinson, \textit{Fractals and self-similarity}, Indiana Univ. Math. J. 30 (1981), no. 5, 713--747.
\bibitem{JorSah} T. Jordan, T. Sahlsten, \textit{Fourier transforms of Gibbs measures for the Gauss map }, Math. Ann. 364(3-4), 983--1023, 2016.
\bibitem{Kau} R. Kaufman, \textit{Continued fractions and Fourier transforms,} Mathematika 27(2), 262--267 (1980). 
\bibitem{Kau2} R. Kaufman, O\textit{n the theorem of Jarník and Besicovitch,} Acta Arith. 39 (1981), no. 3, 265--267.
\bibitem{Koks} J. F. Koksma, \textit{Ein mengentheoretischer Satz über die Gleichverteilung modulo Eins,} Compositio Math. 2 (1935), 250--258. 
\bibitem{KN} L. Kuipers, H. Niederreiter, \textit{Uniform distribution of sequences,}
Wiley-Interscience, John Wiley \& Sons, New York-London-Sydney, 1974. 
\bibitem{Pisot2} C. Pisot, \textit{La r\'{e}partition modulo 1 et les nombres alg\'{e}briques,}
Ann. Scuola Norm. Sup. Pisa Cl. Sci. (2) {\textbf{7}} (1938), no. 3--4, 205--248.

\bibitem{Pisot}
C. Pisot, \textit{Sur la r\'{e}partition modulo $1$ des puissances successives d'un m\^{e}me nombre,}
C.R. Acad. Sci. Paris  {\textbf{204}} (1937),
312--314.
\bibitem{Pol} A. Pollington, S. Velani, A Zafeiropoulos, E. Zorin, \textit{Inhomogeneous Diophantine Approximation on $M_{0}$-sets with restricted denominators,} arXiv:1906.01151.
\bibitem{QR} M. Queff\'{e}lec, O. Ramar\'{e}, \textit{Analyse de Fourier des fractions continues à quotients restreints,} Enseign. Math. (2) 49 (2003), no. 3-4, 335--356.
\bibitem{Sch} W. Schmidt, \textit{On normal numbers,} Pacific J. Math. 10 (1960), 661--672.
\bibitem{SimWei} D. Simmons, B. Weiss, \textit{Random walks on homogeneous spaces and Diophantine approximation on fractals,}
Invent. Math. 216 (2019), no. 2, 337--394.
\bibitem{Weyl} H. Weyl, \textit{\"{U}ber die Gleichverteilung von Zahlen mod. Eins,}
Math. Ann. 77 (1916), no. 3, 313--352.
\end{thebibliography}
\end{document}